\documentclass[a4paper,11pt]{article}
\usepackage[utf8]{inputenc}

\usepackage{amsmath,amssymb,amsthm}

\usepackage{url}
\usepackage{hyperref}
\usepackage{cite}

\usepackage{color}

\newtheorem{theorem}{Theorem}[section]
\newtheorem{lemma}[theorem]{Lemma}
\newtheorem{definition}[theorem]{Definition}
\newtheorem{assumption}[theorem]{Assumption}
\newtheorem{prop}[theorem]{Proposition}
\newtheorem{remark}[theorem]{Remark}

\renewcommand{\epsilon}{\varepsilon}
\newcommand{\eps}{\varepsilon}
\newcommand{\Res}{\operatorname{Res}}

\usepackage[a4paper]{geometry}

\usepackage{todonotes}

\newcommand{\Qcal} {{\mathcal Q}}

\newcommand{\Z}{\mathbb{Z}}
\newcommand{\R}{\mathbb{R}}

\renewcommand{\P}{\mathbb{P}}
\newcommand{\E}{\mathbb{E}}

\numberwithin{equation}{section}
\AtBeginDocument{}

\begin{document}

\title{Singular limits for stochastic equations\footnotemark[1]}
\author{Dirk Bl\"omker\footnotemark[2]~~and~Jonas M. T\"olle\footnotemark[3]}

\maketitle

\footnotetext[1]{This work is licensed under the Creative Commons Attribution 4.0 International License. To view a copy of this license, visit \url{http://creativecommons.org/licenses/by/4.0/} or send a letter to Creative Commons, PO Box 1866, Mountain View, CA 94042, USA. \includegraphics[height=1em]{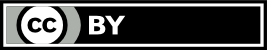}
Original work published in Stochastics and Dynamics, (2023), \url{https://doi.org/10.1142/S0219493723500405}.}

\footnotetext[2]{Universit\"at Augsburg, Institut f\"ur Mathematik, Universit\"atsstra\ss{}e 14, 86135 Augsburg, Germany, \url{dirk.bloemker@math.uni-augsburg.de}.
DB was partially funded by the Deutsche Forschungsgemeinschaft (DFG, German
Research Foundation) -- 514726621.}

\footnotetext[3]{Aalto University, Department of Mathematics and Systems Analysis, PO Box 11100 (Otakaari 1, Espoo), 00076 Aalto, Finland, \url{jonas.tolle@aalto.fi}.\\
JMT acknowledges support by the Academy of Finland and the European Research Council (ERC) under the European Union's Horizon 2020 research and innovation programme (grant agreements no. 741487 and no. 818437).}

\begin{abstract}
We study singular limits of stochastic evolution equations in the interplay of disappearing strength of the noise and insufficient regularity, where the equation in the limit with noise would not be defined due to lack of regularity. 

We recover previously known results on vanishing small noise with increasing roughness, but our main focus is to study for fixed noise  the singular limit where the leading order differential operator in the equation may vanish. Although the noise is disappearing in the limit, additional deterministic terms appear due to renormalization effects.
We separate the analysis of the equation from the convergence of stochastic terms and give a general framework for the main error estimates. This first reduces the result to bounds on a residual and in a second step to various bounds on the stochastic convolution.

Moreover, as examples 
we apply our result to the a singularly regularized Allen-Cahn equation with a vanishing Bilaplacian, and the Cahn-Hilliard/Allen-Cahn homotopy with space-time white noise in two spatial dimensions.
\end{abstract}



\section{Introduction}

Various singular limits 
for stochastic evolution equations appear in many different examples as in the sharp interface limit for Allen-Cahn (AC) or Cahn-Hilliard (CH) with small noise \cite{ABK:12, Funaki:16} or in the approximation and existence theory of singular stochastic partial differential equations (singular SPDEs) \cite{Ha:14b, GIP:15, DPD03}.
Other variants of stochastic singular limits are considered for instance in the study of stochastic slow-fast dynamical systems \cite{BG:06}, two-scale stochastic optimal control problems \cite{GT:22} and stochastic modulation equations \cite{BH:04}.
In the present work, we are especially interested in the case where the deterministic drift has vanishing parts that lead to a lack of regularity in the limit.  

The abstract setting which we will study is an SPDE of the type 
\begin{equation}
 \label{e:abs}
 \partial_t u_\eps = A_\eps  u_\eps + F_\eps(u_\eps) +  Q_\eps \partial_t  W
\end{equation}
for some linear operator $A_\eps$ and a nonlinearity $F_\eps$.
The noise term is always given by an additive Gaussian space-time white noise consisting of the derivative 
of a standard cylindrical Wiener process $W$ and the noise coefficient $Q_\eps$, which is a linear operator.
The covariance operator $Q_\epsilon^{\ast} Q_\eps$
encodes all the information about correlations and regularity of the noise.

We will study \eqref{e:abs} in the limit $\eps\to0$,
where we assume a vanishing noise strength $Q_\eps\to 0$, but we also allow for a possibly
increasing spatial roughness of the noise,
in combination with  the convergence of the linear operators 
$A_\eps\to A$ and of the nonlinearities $F_\eps\to F$ in a certain sense, where leading order differential operators may vanish in the limit.

The main aim of the paper is to find 
interesting asymptotics of the noise coefficient $Q_\eps$ 
such that the limit $u$ of the $u_\eps$   unexpectedly satisfies
\begin{equation*}
 \partial_t u =  A  u +G(u)
\end{equation*}
with $G\not=F$. This is typically the case when the noise strength vanishes in the limit, but the noise is to rough to make sense of the limiting equation. We comment below on existing results that treat exactly this situation, but all of the existing cases in the literature treat a fixed PDE, i.e., $A_\eps$ and $F_\eps$ are independent of $\eps$. In our result we want to allow for fixed roughness of the noise (for example the covariance operator $Q_\eps$ is just a scalar multiple of the identity) and vanishing  regularization by higher order operators disappearing in the limit.  

In the existence theory of semilinear SPDEs, the convergence of approximations by solutions to regularized equations needs to utilize probabilistic cancellations in order to overcome the pathwise irregularity generated by the Gaussian noise. This leads to a renormalization procedure that generates deterministic counterterms for the limit equation as a new requirement for well-posedness of the stochastic equation. This has been systematically studied in \cite{Ha:14a,BCCH:21}.

In order to manage the convergence of the stochastic remainder terms in our approach, we shall in our examples employ a straightforward Fourier series computation  which is well-known in the case of periodic boundary conditions \cite {GIP:15} and the idea of which can be originally credited to Da Prato and Debussche \cite{DPD02,DPD03}, see Section \ref{sec:stochastic-convolution} for details. Our nonlinear estimates will nevertheless hold in a more general framework, but there we only reduce the key error estimate to various bounds on the stochastic convolution, which still have to be verified for the given domain and the boundary conditions. 

We will provide an abstract setting and a fairly general result  and apply this to three prototypical examples.
Our first example mainly recovers the well known result of  
Hairer, Ryser and Weber \cite{HRW:12}, where they studied the vanishing noise limit for Allen-Cahn equations in two spatial dimensions with regularized noise. This is not possible without using renormalized nonlinearities with diverging constants.
Although this application is not new, we include it to illustrate our result, where this would correspond to the case where both $A_\eps$ and $F_\eps$ are independent of $\eps$, and only the noise changes with $\eps$.

In contrast to that our two main examples treat the case of vanishing higher order differential operators where the noise has a fixed roughness. 
First for the Allen-Cahn equation with space-time white noise we study the regularization of $A_\eps$ 
with an additional Bilaplacian that vanishes in the limit. A main difficulty in our approach is that we had to study error estimates in spaces with $\epsilon$-dependent norms, which is not necessary in our first example from  \cite{HRW:12}.

In our third example we will study an
Allen-Cahn/Cahn-Hilliard homotopy, where again the regularity of the noise is fixed, but both $A_\eps$ and the nonlinearity $F_\eps$ have vanishing higher order differential operators, which leads to additional difficulties.

Let us point out that similar results in the same spirit 
of \cite{HRW:12} with vanishing noise of increasing roughness were obtained in the literature, but they do not fit into our framework. 
First \cite{OOR:20} treats a result for the stochastic nonlinear wave equation. Moreover Flandoli, Galeati, and Luo \cite{FGL:21-1, FGL:21-2} studied the effect of transport noise on the Navier-Stokes and Euler equation, where in the limit of vanishing strength  due to increasing roughness of the noise one obtains in the limit a deterministic Navier-Stokes equation with a changed viscosity that depends on the noise. 

Let us mention that vanishing noise limits occur also when studying the large deviations principle and small noise asymptotics for semilinear SPDEs, see e.g. \cite{FJL:82, Zabczyk:89, DPZ:88, HR:31} and the references therein. An interesting result is \cite{HW:15}, where the authors study the large deviation limit in combination with the loss of regularity limit. An asymptotic coupling theorem for the stochastic Allen-Cahn equation with small space-time white noise in two dimensions was studied in \cite{TW:20}.

The dependence of solutions to SPDEs on the convergence of coefficients or parameters in the drift and covariance terms has e.g. been studied in \cite{KvN:11, BCF:88, GT:16, BMSS:95}.

Let us finally remark that regularity structures were studied for equations with Neumann boundary conditions on a square in \cite{GH:19, HP:21} or  
on Riemannian manifolds without boundary in \cite{BB:16}, but the case of general domains 
seems to be open. Nevertheless, we will formulate our results also for the general domain case, as we can reduce all nonlinear estimates to simple bounds on the stochastic convolution, which still need to be verified. We shall only briefly comment on the periodic case in the end.

\subsection*{Organization of the paper}

In Section \ref{sec:results} we shall present our main results divided into the cases of our prototypical examples which are introduced as well. In Section \ref{sec:abstract-approach} we formulate and prove our abstract main Theorem \ref{thm:main} together with its assumptions.
Here the error estimate is carried over to bounds on the residual. In Section \ref{sec:nonlinear} the assumptions are verified for our examples. 
The bound on the residual terms for each example is proved in Section \ref{sec:residual}. Here we carry over all error estimates to bounds on the stochastic convolution only. Finally, our main approximation result is completed by examples for the estimates needed for the stochastic convolution in Section \ref{sec:stochastic-convolution}.
Here we briefly comment on the results in the periodic case, which are direct applications of well known results via Fourier series expansion.

\section{Results and examples}\label{sec:results}

We are interested in the following 
examples which will be covered by our abstract setting. For simplicity of presentation we study a cubic nonlinearity, and differential operators that are diagonal in Fourier space, although our abstract result would allow for a far more general setting.

Moreover, the limiting equations contain only a Laplacian $\Delta$ as a leading order term, which is not enough to give a meaning to the equation perturbed by space-time white noise $\partial_t W$ in two spatial dimensions or more.

\subsection{Cahn-Hilliard/Allen-Cahn homotopy in 2D}

A typical example is the stochastic 
Cahn-Hilliard/Allen-Cahn (CH/AC) equation 
for $\epsilon \in(0,1)$ in the limit $\epsilon\to 0$ 
\begin{equation}
 \label{e:main}
\partial_t u_\eps = ( 1-\epsilon - \epsilon \Delta )(\Delta u_\eps + f(u_\eps))  + \sigma_\eps \partial_t W
\end{equation}
for a standard cubic nonlinearity $f(u)=u-u^3$ with space-time white noise $\partial_t W$
in a two dimensional domain subject to periodic boundary conditions on the flat torus or Neumann boundary conditions on general domains.
In a similar way, we could also consider the equation in three spatial dimensions, but abstain from doing so here. 

From the general theory of SPDEs, one is expecting 
that in terms of existence and regularity of 
solutions this equation for $\eps\in(0,1)$ 
behaves similarly to the standard stochastic Cahn-Hilliard equation 
(case $\eps=1$) --- even with space-time white noise. 
This equation has unique global solutions for initial conditions in $L^2$. 
This can be proved in the same way as for the Cahn-Hilliard equation, treated by Da Prato and Debussche \cite{DPD96}, who were applying the spectral Galerkin method.
In the case of multiplicative noise
the existence and uniqueness of solutions was studied in \cite{AKM:16}.

The main difficulty is that for the convergence $\sigma_\eps\to \sigma>0$ for
$\eps\to0$ one would obtain the stochastic Allen-Cahn equation in the limit, 
which is no longer well-posed in dimension $d\geq2$.
One would need to add diverging renormalization constants to give sense to this case. See for example \cite{HW:15, HRW:12} or \cite{B:19} among others.

For the main result, we try to determine $\sigma_\eps$ in such a way that
with high probability 
\[u_\eps \approx u+Z_\eps\]
where $ Z_\eps$ is the stochastic convolution, i.e.\ the solution to \eqref{e:main} with $f=0$, see \eqref{def:OU} for the definition, and $u$ solves the deterministic PDE
\[\partial_t u = \Delta u+f(u)-3C_0 u,
\] 
where the constant $C_0$ depends on the convergence of sequence $\sigma_\varepsilon$ and is defined in \eqref{def:C0}. In fact, it is the limit in quadratic mean of $Z_\eps^2$ in $H^{-1}$, see \eqref{eq:Zsquare}.

\begin{remark}
In our work we mainly focus on the nonlinear estimates and establish the full approximation result for periodic boundary conditions only. For Neumann boundary conditions, our main results reduce the approximation result to a statement about the stochastic convolution. To evaluate this in full generality seems to be an open problem on general domains.
\end{remark}

Note that for $\sigma_\eps$ too small, 
we are just in a large deviation type regime,
where $C_0=0$. On the other hand, if $\sigma_\eps$ is too large, we are in the renormalization regime, where $C_0$ has to be replaced by an $\eps$-dependent constant that diverges for $\eps\to0$.

\begin{remark}
Let us remark that we do neither need regularity structures, nor paracontrolled distributions for our result. 
Our limit is a deterministic PDE and thus we can assume more regularity of the limit $u$. Therefore, mixed terms like $u Z_\eps^2$ are always well-defined in the space where $Z_\eps^2$ is well-defined.
\end{remark}

\subsection{Allen-Cahn with higher order regularization}

Another example in a similar spirit as the Allen-Cahn equation with regularized noise
is the singular limit with noise subject to the higher order regularization
\begin{equation}
\label{e:regAC}
\partial_t u_\eps =  -\eps^2\Delta^2u_\eps+ \Delta u_\eps + f(u_\eps) +  \sigma_\eps \partial_t W,
\end{equation}
either with periodic or Neumann boundary conditions.
Again, the higher order differential equation is regularizing 
and, in the limit $\eps\to 0$, the equation might not be well-defined if $\sigma_\eps$ converges to a positive constant or is just not vanishing fast enough.

The main results in this case are analogous to the CH/AC-homotopy, but simpler, as we can use less conditions on the stochastic convolution to bound the residual in the main error estimate. 

\subsection{Allen-Cahn with regularized noise}

Hairer, Ryser and Weber \cite{HRW:12} studied AC in two spatial dimensions with regularized noise
subject to periodic boundary conditions, i.e. 
\begin{equation}\label{eq:HRW}
\partial_t u_\eps =  \Delta u_\eps + f(u_\eps) +  \sigma_\eps \Qcal_\eps \partial_t W,
\end{equation}
where the operator $Q_\eps= \sigma_\eps \Qcal_\eps$ is split explicitly into the scalar noise strength $\sigma_\eps$ and spatial correlation given by the operator $\Qcal_\eps$.
Here the equation is fixed, 
but due to the regularization $\Qcal_\eps\to I$ for constant $\sigma_\eps = \sigma$  the limit $\eps\to 0$ 
is not possible without using renormalized nonlinearities with diverging constants.

The authors also identified three regimes. For fixed (or too large) noise strength the solution $u_\eps$ would converge to zero.
For too small noise strength the noise would just vanish, and in the intermediate regime, one obtains a non-trivial limit.

An interesting result is \cite{HW:15} where the authors study the limits of noise strength to zero and loss of regularity of the noise both separately and combined. 
Our method partially
recovers their result for the joint limit when the limiting equation is a deterministic PDE. 

\subsection{Viscous Cahn-Hilliard equation}

A model related to the Cahn-Hilliard/Allen-Cahn homotopy
is the viscous Cahn-Hilliard (CH) equation introduced by Novick-Cohen
\cite{NC:88},
that can be transformed into 
\[
\partial_t u_\eps = - ( 1- \alpha - \alpha \Delta)^{-1} \Delta (\epsilon^2\Delta u_\eps + f(u_\eps))
+ \sigma_\eps \partial_t W.
\]
Here the case $\alpha=1$ corresponds to the AC equation,
while $\alpha=0$ is the CH equation.

But this example does not fit into our setting, 
as with space-time white noise 
only the case $ \alpha=0$ is well-posed in spatial dimension $d=2$ and $d=3$,
while for $\alpha\in (0,1]$, we would need renormalization as the linear operator 
does not generate enough regularity of the solution.

\section{Abstract approach}\label{sec:abstract-approach}

Recall our abstract equation \eqref{e:abs} posed in a separable Hilbert space $H$.
The norm of $H$ is denoted by $\|\cdot\|$.
In the abstract result, we will not go into the detail 
of establishing the existence of solutions, 
but comment in our examples on it in more detail.
For a general approach to well-posedness of SPDEs see \cite{DP-Z:14, LR:15, DPZ:88}.

\begin{assumption}
\label{ass:space}
Suppose that we have for all $\eps>0$ a Gelfand-triple 
together with an additional Banach space $X$ such that
\[
V_\eps \subset X \subset H \simeq H'\subset V_\eps'
\]
for a separable and reflexive Banach space $V_\eps$ with 
topological dual $V_\eps'$ such that all embeddings are continuous and dense.
\end{assumption}

Let us define the stochastic convolution  
\begin{equation}
\label{def:OU}
Z_\eps(t) = \int_0^t e^{(t-s)A_\eps } Q_\eps  \,dW(s)
\end{equation}
which is in general defined in terms of a 
$C_0$-semigroup $e^{tA_\eps}$ generated by $A_\epsilon$. 
Here we first suppose that $Z_\eps$ is a 
well-defined stochastic process in $X$, more precisely, assume:
\begin{assumption}
\label{ass:Z}
The process 
$Z_\eps$ belongs $\P$-a.s. to $C^0([0,T];X)$ for every $\eps>0$ and $T>0$.
\end{assumption}
We will later add additional assumptions on $Z_\eps$.
In general, we need  $Z_\eps$ to take values in a more regular
space than $H$, as $F_\eps$ is in our examples in general not defined 
on $H$ but on a smaller space $X$. See Assumption \ref{ass:AF} below.
 
Using the standard transformation,
we define
\[
v_\eps := u_\eps - Z_\eps\;,
\]
in order to obtain 
\begin{equation}
 \label{e:trafo}
 \partial_t v_\eps =  A_\eps v_\eps + F_\eps(v_\eps+Z_\eps).
\end{equation}

\begin{remark} 
As we need
at least some regularity of $u$ and $v_\eps$ for the singular limit, 
even at time zero, 
we do not treat the substitution with a stationary stochastic convolution, 
where the integral in $Z_\eps$ starts at $-\infty$, 
which would be more natural in this context. 
\end{remark}

We want to show that if $Z_\eps\to 0 $ and $Z_\eps^2\to C_0$ in a certain sense (in fact, with values in $X$ both limits will diverge) 
then we can still
have averaging/renormalization effects in $F(v_\eps+Z_\eps)$ that appear in the limit and lead to additional terms.

Thus we consider for some nonlinearity $G$ the limiting PDE
\begin{equation}
\label{e:limit}
 \partial_t u =  A  u+ G(u). 
\end{equation}

\begin{definition}\label{def:weak}
We say that $u$ is a weak solution to \eqref{e:limit} with initial datum $u(0)=u_0\in H$ in a Gelfand triple
\[V \subset X \subset H \simeq H'\subset V'\]
with dense and continuous embeddings for measurable operators $G:X\to V'$ and $A:V\to V'$ if $u\in L^2([0,T],V)$ such that for all $t\in [0,T]$
\[
\langle u(t), v \rangle_H= 
\langle u_0,v\rangle_H
+\int_0^t \langle Au(s),v\rangle\,ds
+\int_0^t \langle G(u(s)),v\rangle\,ds
\]
for every $v\in V$, and all integrals are well defined.
\end{definition}
Note that angle brackets $\langle\cdot,\cdot\rangle$ in the previous definition 
denote the dual pairing between $V$ and $V'$ induced by the identification of $H$ with its dual $H'$.

\begin{definition}
We say that $u_\eps$ solves \eqref{e:abs} if $v_\eps=u_\eps - Z_\eps$ is a weak solution to \eqref{e:trafo}, where $Z_\eps$ denotes the stochastic convolution from Assumption \ref{ass:Z}, and where we define weak solutions to \eqref{e:trafo} as in Definition \ref{def:weak}, where $V$ is replaced by $V_\eps$, $A$ is replaced by $A_\eps$, and $G$ is replaced by $F_\eps$.\end{definition}

\begin{assumption}\label{ass:ex}
Let $G:X\to V'$, $F_\eps:X\to V_\eps'$, $A:V\to V'$, and $A_\eps:V_\eps\to V_\eps'$ be measurable operators.
We assume that weak solutions to \eqref{e:limit} exist in a Gelfand triple
\[V \subset X \subset H \simeq H'\subset V'\]
with dense and continuous embeddings such that
\begin{equation}\label{eq:ureg}
u\in L^2([0,T],V) \cap C^0([0,T],X),\quad \text{and}\quad\partial_t u\in L^2([0,T],V').
\end{equation}
Moreover, we assume that  weak solutions to \eqref{e:trafo} exist in the Gelfand triple
\[V_\eps \subset X \subset H \simeq H'\subset V_\eps'\]
such that
\begin{equation}\label{eq:vepsreg}
v_\eps \in L^2([0,T],V_\eps) \cap C^0([0,T],X),\quad\text{and}\quad \partial_t v_\eps\in L^2([0,T],V_\eps').
\end{equation}
Suppose moreover that the error $\varphi_\eps=u-v_\eps$ is well-defined with
\begin{equation}\label{eq:varphireg}
\varphi_\eps\in L^2([0,T],V_\eps),\quad \text{and}\quad\partial_t \varphi_\eps\in L^2([0,T],V_\eps').
\end{equation}
\end{assumption}

\begin{remark}
In order to ensure the existence of solutions, 
we of course need additional assumptions on the nonlinearities $F_\eps$ and $G$.
We comment below in all our examples that it is usually straightforward to ensure that Assumption \ref{ass:ex} is true.
Moreover, it is usually easy to verify in examples that the solution $u$ of \eqref{e:limit} becomes more regular than in \eqref{eq:ureg} provided the initial condition $u(0)=u_0$ is more regular than merely $u_0\in H$.
\end{remark}

Note that the assumption \eqref{eq:varphireg} implies $\varphi_\eps\in C^0([0,T],H)$ and is moreover
sufficient for all the energy estimates,
as in this case $\|\varphi_\eps\|^2 \in W^{1,1}([0,T])$ 
with $\partial_t \|\varphi_\eps\|^2=2\langle \partial_t \varphi_\eps,\varphi_\eps\rangle$, see e.g. \cite[Proposition III.1.2]{Show}.

The assumption \eqref{eq:varphireg} requires usually higher regularity of $u$ as just \eqref{eq:ureg}. In our examples, we have a situation that $V_\eps\subset V_{\tilde{\eps}}\subset V$, for any $\eps>\tilde{\eps}>0$. We may assume as much regularity for the initial datum $u_0$ of \eqref{e:limit} as needed such that $u\in L^2([0,T],V_\eps)$ and $\partial_t u\in L^2[0,T],V'_\eps)$ for every $\eps>0$. We shall not assume any abstract relation of $V_\eps$ and $V$ as the spaces are usually canonically attached to $A_\eps$ and $A$ respectively.

Our approach is based on residual estimates and an approximation result, 
where there are basically two possible ways.
One needs to plug the solution to one of the equations for the residual estimate 
in the other equation, which yields an approximation result.

We shall calculate the residual of $u$, if plugged into (\ref{e:trafo}).
Define
\begin{eqnarray}
\Res_\eps(u)(t) \label{def:resu}
& :=& \partial_t u(t)-A_\eps u(t) -F_\eps(u(t)+ Z_\eps(t)) \\
& =& (A-A_\eps) u(t) + F(u(t)+ Z_\eps(t)) - F_\eps(u(t)+ Z_\eps(t)) 
\nonumber\\&& +G(u(t))-F(u(t)+ Z_\eps(t))
\nonumber
\end{eqnarray}
Note that in the abstract setting, in general, we do not know 
whether the residual is defined and sufficiently regular.
This has to be checked in applications, when the bound on the residual necessary for the main theorem have to be checked. Concerning our examples, we shall verify this in Section
\ref{sec:residual}.

The advantage of the latter rewriting is that the term $F(u(t)+ Z_\eps(t))$, 
where the renormalization terms will appear, has a fixed nonlinearity $F$,
which simplifies the argument significantly.

The aim is now to choose spaces $V_\epsilon \subset H$ 
such that we can bound the residual in the dual space $V'_\epsilon$. We will always use the convention that the  $V'_\epsilon$-norm of the residual is infinite if it is not well-defined in that space.

Note that in some examples the spaces $V_\eps$ and thus $V_\eps'$ are independent of $\epsilon>0$, 
but for instance in the Cahn-Hilliard/Allen-Cahn homotopy, the $\eps$-dependence is natural.

The crucial assumption that hides all the
technicalities in the estimates for the 
nonlinear terms  is:

\begin{assumption}
 \label{ass:AF}
 Consider the spaces $ X $ and $V_\eps$ from Assumption \ref{ass:space}
 and suppose $F_\eps:X\to V_\eps'$ and $A_\eps : V_\eps\to V_\eps'$.
Assume that there exist constants $C\ge 0$, $\delta>0$, and $0\leq c_\eps \to 0$ for $\eps\to0$,
such that:
\begin{equation}\label{eq:main_estimate}
  \langle A_\eps \varphi,\varphi\rangle 
  + \langle F_\eps(\varphi+\psi) 
  - F_\eps(\psi), \varphi\rangle
  \leq 
  - \delta\|\varphi\|_{V_\eps}^2 + C\|\varphi\|^2
\end{equation}
for every choice of $\varphi\in V_\eps$ and $\psi\in X$ 
with $c_\eps \|\psi\|_X \leq 1$.
\end{assumption}

We are ready to state and prove our main theorem.

\begin{theorem}
\label{thm:main}
Under Assumptions \ref{ass:space}, \ref{ass:ex}, and  \ref{ass:AF} 
let $u$, $u_\eps$ be any two solutions to \eqref{e:limit}, \eqref{e:abs}, respectively, from Assumption \ref{ass:ex}
and define by $Z_\eps$ the stochastic convolution from Assumption \ref{ass:Z}.

Then,
for  all $T>0$, there is a constant $K>0$ such that  
for $\epsilon>0$ sufficiently small, we have that 
for all $\gamma>0$ 
\begin{eqnarray*}
\mathbb{P} \left( \sup_{[0,T]}\|u_\eps-u-Z_\eps\|^2 >  K \gamma\right)
&\leq &
\mathbb{P} \left(\sup_{[0,T]}  \|Z_\eps \|_X > (2c_\eps)^{-1} \right)+\mathbb{P} \left(\|u(0)-u_\eps(0)\|^2 > \gamma\right)\\
&&
+\mathbb{P} \left(\int_0^T\| \Res_\eps(u)\|^2_{V_\eps'}\,dt > \gamma\right)
\end{eqnarray*}
\end{theorem}

Note that for $c_\eps>0$ the condition $\epsilon>0$ being sufficiently small
in the previous theorem can be quantified by
$\sup_{[0,T]}  \|u \|_X \leq (2c_\eps)^{-1}$, see \eqref{eq:eps_small}.

In all cases where $c_\eps=0$ we 
will use that $1/c_\eps=\infty$ and no smallness of $\eps$ is needed and furthermore
no condition on the $X$-norm of $Z_\eps$ is needed.

\begin{proof}[Proof of Theorem \ref{thm:main}]
For the approximation result consider the error with $u,v_\eps$ from Assumption \ref{ass:ex} 
\[
\varphi_\eps = v_\eps - u=u_\eps-Z_\eps-u
\]
which (using the residual) solves
\begin{equation}
 \label{e:error}
 \partial_t \varphi_\eps =  A_\eps \varphi_\eps +  F_\eps(\varphi_\eps+u+Z_\eps) - F_\eps(u+Z_\eps)  + \Res_\eps(u)\;.
\end{equation}
Now based on the regularity assumed on $\varphi_\eps$ we use standard a priori estimates based on 
\[ 
\frac12\partial_t \|\varphi_\eps\|^2 
= \langle \partial_t \varphi_\eps,\varphi_\eps\rangle 
\]
to obtain
\begin{equation}
 \label{e:aprio}
\frac12\partial_t \|\varphi_\eps\|^2 = 
\langle A_\eps \varphi_\eps,\varphi_\eps\rangle 
+ \langle F_\eps(\varphi_\eps+u+Z_\eps) 
- F_\eps(u+Z_\eps) , \varphi_\eps \rangle 
+ \langle \Res_\eps(u),\varphi_\eps\rangle .
\end{equation}
 
We now use the crucial assumption \ref{ass:AF} 
that hides all the technical estimates and assume that 
\begin{equation}
\label{e:crucial}
 c_\eps\|u+Z_\eps\|_X \leq 1.
\end{equation}
Note that as $u$ is  independent  of $\epsilon$ we have that for 
small $\epsilon>0$ such that
\begin{equation}\label{eq:eps_small}c_\eps\|u\|_X \leq 1/2\end{equation}
the equation \eqref{e:crucial} is true provided  $c_\eps\|Z_\eps\|_X \leq 1/2$,
which we will assume from now on.
Moreover, recall that $u\in C^0([0,T],X)$ by assumption.
Thus for $\eps>0$ sufficiently small in the sense of \eqref{eq:eps_small},
\[\sup_{[0,T]}\|Z_\eps\|_X\le (2c_\eps)^{-1}\]
implies
\[ \sup_{[0,T]}\|u+Z_\eps\|_X \leq c_\eps^{-1}.\]
Using \eqref{e:crucial} and Assumption \ref{ass:AF} we obtain from \eqref{e:aprio} 
\[
\partial_t\|\varphi_\eps\|^2  \leq  
- 2\delta\|\varphi_\eps\|_{V_\eps}^2 
+ 2 C\|\varphi_\eps\|^2
+ 2\|\varphi_\eps\|_{V_\eps} \|\Res_\eps(u)\|_{V_\eps'} 
\]
Thus by Young inequality 
\[
\partial_t\|\varphi_\eps\|^2  \leq  
- \delta\|\varphi_\eps\|_{V_\eps}^2 
+ 2C\|\varphi_\eps\|^2 
 + \frac1\delta \|\Res_\eps(u)\|^2_{V_\eps'} \;.
\]
By Gronwall type arguments we obtain for $t\in[0,T]$
\[
 \|\varphi_\eps(t)\|^2  \leq 
 \|\varphi_\eps(0)\|^2e^{2Ct} + 
 \frac1\delta  \int_0^t  
  e^{2C (t-s)} \|\Res_\eps(u)\|^2_{V_\eps'} \,ds   \;.
\]
Now we assume that 
\[
 \int_0^T\| \Res_\eps(u)\|^2_{V_\eps'}\, ds  \leq \gamma
 \quad \text{and}\quad 
  \|\varphi_\eps(0)\|^2 \leq \gamma.
\]
Thus we obtain a constant $K>0$ depending only on $\delta$ and $T>0$ such that
\[ \sup_{t\in[0,T]}\|\varphi_\eps(t)\|^2  \leq K\gamma \;.
\]
\end{proof}

The main task for the remaining parts of the paper is to verify Assumption \ref{ass:AF} in the given examples 
and to show that the bound on the residual needed in the previous theorem is actually small. 

The initial conditions will always satisfy $\mathbb{P} \left(\|u(0)-u_\eps(0)\|^2 > \gamma\right)=0$ for all $\gamma>0$ if we assume $u(0)=u_\eps(0)$. However, we will need solutions $u$ of the deterministic PDE \eqref{e:limit} to be more regular than \eqref{eq:ureg}, so one can think of $u(0)$ as a regularized version of $u_\eps(0)$.

\section{The nonlinear estimate}\label{sec:nonlinear}

In this section we check Assumption \ref{ass:AF} in our examples
and determine the spaces $X$, $H$, $V$, and $V_\eps$.

In what follows in this section, let $H=L^2(U)=L^2(U,dx)$ be the standard $L^2$-space, where either $U\subset\R^2$ is a bounded domain with sufficiently smooth boundary
or $U=\mathbb{T}^2=\R^2/(2\pi\Z^2)$.
We will check the assumptions of the abstract setting both in the periodic case and in the Neumann case, but to keep the presentation concise, we discuss the bounds on the stochastic convolution necessary to establish the full result very briefly in the periodic case only.   

We will not discuss the  existence and uniqueness of solutions for the three SPDEs below in full detail. For these type of equations this is fairly standard. 
See for example \cite{DPD96} for the Cahn-Hilliard equation, which could easily be adapted to the CH/AC-homotopy in the next section.
For stochastic Allen-Cahn we refer to the lecture notes \cite{B:19}, among many others, which also adapts to the case of the regularized AC equation.

For the convenience of the reader, we collect the definition of spaces and operators for all of our three examples in the table below.

\bigskip

\begin{tabular}{|c|c|c|c|}
\hline 
Symbol & CH/AC-homotopy & AC higher order reg. & AC reg. noise\tabularnewline
\hline 
\hline 
$V$ & $H^{1}$ & $H^{1}$ & $H^{1}$\tabularnewline
\hline 
$X$ & $C^{0}$ & $C^{0}$ & $C^{0}$\tabularnewline
\hline 
$H$ & $L^{2}$ & $L^{2}$ & $L^{2}$\tabularnewline
\hline 
$A_{\varepsilon}$ & $(1-\varepsilon-\varepsilon\Delta)\Delta$ & $-\varepsilon^{2}\Delta^{2}+\Delta$ & $\Delta$\tabularnewline
\hline 
$V_{\varepsilon}$ & $\left(H^{2},\|(1-A_{\varepsilon}^{1/2})(\cdot)\|_{L^{2}}\right)$ & $\left(H^{2},\|(1-A_{\varepsilon}^{1/2})(\cdot)\|_{L^{2}}\right)$ & $H^{1}$\tabularnewline
\hline 
$A$ & $\Delta$ & $\Delta$ & $\Delta$\tabularnewline
\hline 
$f(u)$ & $u-u^{3}$ & $u-u^{3}$ & $u-u^{3}$\tabularnewline
\hline 
$F_{\varepsilon}(u)$ & $(1-\varepsilon-\varepsilon\Delta)f(u)$ & $u-u^{3}$ & $u-u^{3}$\tabularnewline
\hline 
$F(u)$ & $u-u^{3}$ & $u-u^{3}$ & $u-u^{3}$\tabularnewline
\hline 
$G(u)$ & $u-u^{3}-3C_{0}u$ & $u-u^{3}-3C_{0}u$ & $u-u^{3}-3C_{0}u$\tabularnewline
\hline 
$Q_{\varepsilon}$ & $\sigma_{\varepsilon}$ & $\sigma_{\varepsilon}$ & $\sigma_{\varepsilon}\mathcal{Q}_{\varepsilon}$\tabularnewline
\hline 
\end{tabular}

\subsection{Cahn-Hilliard/Allen-Cahn homotopy in 2D}

In \eqref{e:main}, we have either with Neumann or periodic boundary conditions
\[
A_\eps = ( 1-\epsilon - \epsilon \Delta )\Delta
\]
and
\[
F_\eps(u) = ( 1-\epsilon - \epsilon \Delta) f(u)  
\]
where we use  
the standard cubic $f(u)=u-u^3$ such that via Young inequality
\[
(f(\varphi+\psi)-f(\psi)) \varphi    
= \varphi^2 -\varphi^4-3\varphi^3\psi-3\varphi^2\psi^2
\leq  \varphi^2 - \frac14 \varphi^4.
\]
Moreover, it is easy to verify that for $\eps \in (0,1/2)$,
we have in $H$, being the standard $L^2$-space, that
\[\langle  A_\eps \varphi,\varphi\rangle_{L^2}
= - \eps \|\Delta \varphi\|^2_{L^2} - (1-\eps) \|\nabla\varphi\|^2_{L^2}
\leq - \frac12  \|\nabla\varphi\|^2_{L^2}.
\]
We choose $V_\eps=H^2$ to be the standard Sobolev space
with the $\eps$ dependent norm given by 
\begin{equation}\label{H2-norm} 
\|\varphi\|_{V_\eps}^2 = \|\varphi\|^2_{L^2}-\langle A_\eps \varphi,\varphi\rangle_{L^2} \;.
\end{equation}
For the nonlinear estimate, we obtain by using the standard cubic
\begin{eqnarray*}
\lefteqn{\langle F_\eps(\varphi+\psi)-F_\eps(\psi) , \varphi\rangle_{L^2}} \\ 
&=&   - \eps \langle   f(\varphi+\psi)-f(\psi), \Delta \varphi \rangle_{L^2} 
+ (1-\eps)   \langle   f(\varphi+\psi)-f(\psi), \varphi\rangle_{L^2}\\
&=& \eps \|\nabla\varphi\|^2_{L^2} + (1-\eps) \|\varphi\|^2_{L^2} + 
\eps \int_U [(\varphi+\psi)^3-\psi^3 ] \Delta \varphi \, dx - (1-\eps) \int_U  [(\varphi+\psi)^3-\psi^3 ] \varphi\, dx.
\end{eqnarray*}
For the cubic terms, we use
\[
\eps \int_U \varphi^3 \Delta \varphi \; dx= -3\eps \int_U \varphi^2 |\nabla\varphi|^2 \, dx \leq 0,
\]
and
\[
\eps \int_U \varphi^2 \psi \Delta \varphi \, dx 
\leq \eps \|\varphi\|_{L^4}^2\|\psi\|_\infty\|\Delta \varphi\| 
\leq 4\eps  \|\varphi\|_{L^4}^4\|\psi\|^2_\infty   + \frac18 \eps \|\Delta \varphi\|^2\;. 
\]
Moreover,
\[
\eps \int_U \varphi \psi^2 \Delta \varphi \, dx 
\leq \eps \|\varphi\|_{L^2}\|\psi\|^2_\infty\|\Delta \varphi\| 
\leq 4\eps \|\varphi\|_{L^2}^2\|\psi\|^4_\infty  
+ \frac18 \eps \|\Delta \varphi\|^2\;. 
\]
and finally
\[
- (1-\eps) \int_U  [(\varphi+\psi)^3-\psi^3 ] \varphi\, dx
\leq - \frac14 (1-\eps)\|\varphi\|_{L^4}^4.
\]
Now,
\[
\langle F_\eps(\varphi+\psi)-F_\eps(\psi) , \varphi\rangle_{L^2} 
\leq  
\frac12\eps\|\Delta \varphi\|^2_{L^2} 
- \left(\frac14-C\eps\|\psi\|_\infty^2\right) \|\varphi\|^4_{L^4}
+ C\left(1+\eps\|\psi\|_\infty^4\right)\|\varphi\|^2_{L^2}.
\]
Finally, for $\psi$ with $\|\psi\|_\infty \leq \eps^{-1/4}$ we have 
\[
\langle F_\eps(\varphi+\psi)-F_\eps(\psi) , \varphi\rangle_{L^2} 
\leq  
\frac12\eps\|\Delta \varphi\|^2_{L^2} 
- \left(\frac14-C\eps^{1/2}\right) \|\varphi\|^4_{L^4}
+ 2C\|\varphi\|^2_{L^2}.
\]
Thus for the space $X=C^0=C^0(\overline{U})$ the assumption \ref{ass:AF} 
is satisfied for  $c_\epsilon=\eps^{1/4}$ and small $\epsilon>0$.
Moreover, note that by Sobolev embedding $\varphi\in L^4$ if $\varphi\in H^2$. 

We can now rewrite our main theorem \ref{thm:main} to obtain:

\begin{theorem}
In the setting of the CH/AC-homotopy in 2D,
let $u$, $u_\eps$ be any two solutions to \eqref{e:limit}, \eqref{e:abs}, respectively, from Assumption \ref{ass:ex}
and define by $Z_\eps$ the stochastic convolution from Assumption \ref{ass:Z}.

Then,
for  all $T>0$ there is a constant $K>0$ such that  
for all $\gamma>0$ 
\begin{eqnarray*}
\mathbb{P} \left( \sup_{[0,T]}\|u_\eps-u-Z_\eps\|^2_{L^2} >  K \gamma\right)
&\leq &
\mathbb{P} \left( \eps^{1/4} \sup_{[0,T]}  \| Z_\eps \|_{C^0} > 1/2 \right) 
\\ &&
+\mathbb{P} \left(\|u(0)-u_\eps(0) 
\|^2_{L^2} > \gamma\right)\\
&& +\mathbb{P} \left(\int_0^T\| \Res_\eps(u)\|^2_{V_\eps'}\,dt > \gamma\right).
\end{eqnarray*}
\end{theorem}

\subsection{Allen-Cahn with higher order regularization in 2D}

In \eqref{e:regAC} we have  either with
Neumann or periodic boundary conditions
\[
A_\eps=-\eps^2\Delta^2+\Delta
\]
and consider it as an unbounded operator on the standard $L^2$ space. Thus we have
\[
\langle A_\eps \varphi,\varphi\rangle 
=
-\eps^2 \|\Delta \varphi\|^2 - \|\nabla\varphi\|^2
\]
and we choose $V_\eps=H^2$ to be the standard Sobolev space 
but with the $\epsilon$-dependent  Norm $(\|\varphi\|^2-\langle A_\eps \varphi,\varphi\rangle)^{1/2}$.

For the nonlinearity we have the standard cubic
\[f(u)=u -u^3
\]
and thus for
\[
F_\eps(u) = f(u)  
\]
we have 
\[
(f(\varphi+\psi)-f(\psi)) \varphi    
\leq  \varphi^2 - \frac14 \varphi^4 .
\]

Hence, Assumption \ref{ass:AF} is satisfied with $c_\eps=0$ for our space $V_\eps$, which also yields that we do not need a condition for the $X$-norm of $Z_\eps$.
We obtain

\begin{theorem}
In the setting of the AC with higher order regularization in 2D,
let $u$, $u_\eps$ be any two solutions to \eqref{e:limit}, \eqref{e:abs}, respectively, from Assumption \ref{ass:ex}
and define by $Z_\eps$ the stochastic convolution from Assumption \ref{ass:Z}.

Then,
for  all $T>0$ there is a constant $K>0$ such that  
for all $\gamma>0$ 
\begin{eqnarray*}
\mathbb{P} \left( \sup_{[0,T]}\|u_\eps-u-Z_\eps\|^2_{L^2} >  K \gamma\right)
&\leq &
\mathbb{P} \left(\|u(0)-u_\eps(0) 
\|^2 > \gamma\right) 
\\&&
+\mathbb{P} \left(\int_0^T\| \Res_\eps(u)\|^2_{V_\eps'}\,dt > \gamma\right)
\end{eqnarray*}
\end{theorem}

It now remains to bound the residual in the $H^{-2}$ space, 
but with the $\epsilon$-dependent norm.

\subsection{Allen-Cahn with regularized noise in 2D}

In \eqref{eq:HRW} we have  either with Neumann or periodic boundary conditions
\[
A_\eps=\Delta
\]
and consider it as an unbounded operator on the standard $L^2$ space. Thus we have
\[\langle A_\eps \varphi,\varphi\rangle = - \|\nabla\varphi\|^2\]
and we choose $V_\eps=H^1$ to be the standard Sobolev space
for all $\eps>0$.

For the nonlinearity we have the standard cubic
\[f(u)=u -u^3
\]
and thus for
\[
F_\eps(u) = f(u)  
\]
we have 
\[
(f(\varphi+\psi)-f(\psi)) \varphi    
\leq  \varphi^2 - \frac14 \varphi^4
\]
Hence, Assumption \ref{ass:AF} is satisfied with $c_\eps=0$ and with $V_\eps=H^1$.

We obtain essentially the same theorem for the regularized AC equation as in the previous section for the higher order regularization.

\section{Convergence of the residual}\label{sec:residual}

In this section, we consider the situation in our examples that the probability of the $L^2([0,T],V_\eps)$-norm of the residual being larger than $\gamma>0$ is small, and bound it by various probabilities that depend only on the stochastic convolution.  
Moreover, we finally fix the $G$ in \eqref{e:limit}.
In all three examples the limiting equation is
\begin{equation}
    \label{e:limitAC}
    \partial_tu=\Delta u + f(u)- 3C_0 u 
\end{equation}
with $G(u)=f(u)-C_0u$, where $C_0$ is defined by the limit of $Z_\eps^2$ for $\eps\to0$.

The limiting equation is a standard Allen-Cahn equation so that both for Neumann or periodic boundary conditions, we obtain a unique solution
\[u\in C^0([0,T],H^1)\cap L^2([0,T],H^2)
\]
for initial conditions $u(0)\in H^1$ via well-known results for nonlinear parabolic PDEs using
for instance spectral Galerkin methods, see e.g. \cite{Sch:2000,T:97}.

Thus we fix a sufficiently smooth solution $u$ and in the following consider the bound on the residual.
Recall the definition \eqref{def:resu} 
\begin{eqnarray*}
\Res_\eps(u)(t)
& =& (A-A_\eps) u(t) + F(u(t)+ Z_\eps(t)) - F_\eps(u(t)+ Z_\eps(t)) 
\nonumber\\&& +G(u(t))-F(u(t)+ Z_\eps(t))
\nonumber
\end{eqnarray*}
which we need by the main theorem to
be bounded in $L^2([0,T],V_\eps')$.

In the case that $A_\eps=A$ and $F_\eps = F$ do not depend on $\eps$, the residual is encoding the information of the renormalization coming from the vanishing noise.

Firstly, we need
\begin{equation}
\label{e:resb1}
\int_0^T\|(A-A_\eps) u\|^2_{V_\eps'}\, dt \to 0
\quad\text{for } \eps\to0. 
\end{equation}
We will see below that this is assured if $u$ is sufficiently regular.

Secondly, we need 
\begin{equation}\label{e:resb2}
\int_0^T\| F(u+Z_\eps) - F_\eps(u+Z_\eps) \|^2_{V_\eps'}\, dt \to 0
\quad\text{for } \eps\to0. 
\end{equation}
This term in \eqref{e:resb2}
is zero for the singular AC,
but for the CH/AC equation more is needed. 
Here,
\[F(u+Z_\eps) - F_\eps(u+Z_\eps) = -\eps(1+\Delta) f(u+Z_\eps) \]
and thus we need bounds on $Z_\eps$ and $u$ 
to get this bounded in $L^2([0,T],V_\eps')$.

Thirdly, consider 
\[
G(u)=f(u) - 3C_0 u
\]
for our simple cubic, to obtain
\[
F(u+Z_\eps) - G(u) =  Z_\eps - 3 u^2 Z_\eps - 3 u [Z_\eps^2-C_0] - Z_\eps^3 .
\]
For these terms, we need various bounds on $Z_\eps$ and averaging results from the next section, together with higher regularity of $u$.

\begin{remark}
Note that the previous statement already focuses on the periodic case, where $C_0$ is a constant. Due to boundary effects, for the case of Neumann boundary conditions on general domains, the situation of $C_0$ not being a constant function in the spatial component may occur, see \cite{GH:19}.
\end{remark}

\subsection{CH/AC-homotopy in 2D}

For \eqref{e:resb1} we have $A=\Delta$ and $A_\eps=(1-\eps-\eps\Delta)\Delta$ and thus 
\[A-A_\eps=\eps(1+\Delta)\Delta
\]
where the norm in $V_\eps$ is 
\[ \|u\|_{V_\eps}^2=
\langle (1+(-\Delta)(1-\eps-\eps\Delta))u,u\rangle.
\]
Suppose that $\{e_k\}$, $k\ge 1$, is an $L^2$-orthonormal basis of eigenfunctions of $-\Delta$ with corresponding eigenvalues $\{\mu_k\}$, $k\ge 1$, then 
for $u=\sum_{k\ge 1} u_k e_k$, we have
\begin{eqnarray*}
\| (A-A_\eps) u \|^2_{V_\eps'} 
&=& \eps^2 \sum_{k\ge 1} \frac{(1-\mu_k)^2\mu_k^2}{(1+\mu_k)(1-\eps+\eps\mu_k)} u_k^2 \\
&\leq& C\eps \sum_{k\ge 1} (1+\mu_k^2) u_k^2 
\\&&
\leq C\eps\|u\|_{H^2}^2.
\end{eqnarray*}
Thus \eqref{e:resb1} holds, provided $u\in L^2([0,T],H^2)$.
We can bound the first term in the residual by $C\eps$.

Note that
\begin{equation}\label{eq:op-bound}
\|1+\Delta\|_{L(L^2,V_\eps')} \leq C\eps^{-1/2},
\end{equation}
which can be seen as follows. For $\eps\in (0,1)$,
\begin{align*}\|1+\Delta\|_{L(L^2,V_\eps')}=&\|(1+\Delta)((1+(-\Delta))(1-\eps-\eps\Delta))^{-1/2}\|_{L(L^2,L^2)}\\
=&\sup_{k\ge 1}\left|\frac{1-\mu_k}{\sqrt{(1+\mu_k)(1-\eps+\eps\mu_k)}}\right|\\
\le&\sup_{k\ge 1}\left|\frac{1-\mu_k}{\sqrt{\eps(1+\mu_k)\mu_k}}\right|\le\sup_{k\ge 1}\left|\frac{\mu_k-1}{\sqrt{\eps}\mu_k}\right|\le \left(1+\frac{1}{\inf_{k\ge 1}\mu_k}\right)\eps^{-1/2}.
\end{align*}
For \eqref{e:resb2}, we rely on \eqref{eq:op-bound}, to bound 
\[
F_\epsilon(w)-F(w) = -\eps(1+\Delta) f(w),
\]
where $\|\cdot\|_{L(L^2,V_\eps')}$ denotes the operator norm.
Thus
\begin{eqnarray*}
\int_0^T\| F(u+Z_\eps) - F_\eps(u+Z_\eps) \|^2_{V_\eps'}\, dt
&\leq& 
C\eps^{1/2}\int_0^T \|f(u+Z_\eps)\|^2_{L^2}dt\\
&\leq & C\eps^{1/2} \left(1+\|u\|^6_{L^6([0,T],L^6)}+\|Z_\eps\|^6_{L^6([0,T],L^6)}\right). 
\end{eqnarray*}
As $u\in L^\infty([0,T],H^1)$, we can bound the first term above by $C\eps^{1/2}$. For the second term, we need a bound on $Z_\eps$ in $L^6([0,T],L^6)$.
 
For the third term 
in the residual we use that 
$\|w\|_{V_\eps'} \leq C \|w\|_{H^{-1}} $ 
to obtain 
\begin{eqnarray*}
\lefteqn{\int_0^T\|F(u+Z_\eps) - G(u) \|^2_{V_\eps'}\, dt}\\
&\leq& C \int_0^T\|F(u+Z_\eps) - G(u) \|^2_{H^{-1}}\, dt \\
&\leq& C \int_0^T\|Z_\eps - 3 u^2 Z_\eps - 3 u [Z_\eps^2-C_0] - Z_\eps^3 \|^2_{H^{-1}}\, dt 
\end{eqnarray*}
Now we need smallness of the bounds on $Z_\eps$ and $Z_\eps^3$ in $ L^2([0,T],H^{-1})$
to control the first and last term small.

For the second and third term, we need to bound products in $H^{-1}$ using for example the following result. 

\begin{lemma}
For $v\in H^{-1}$ and $w\in H^{1\vee(d/2+s)}$, $s>0$ on a $d$-dimensional domain, the pointwise product $v\cdot w$ is defined and satisfies
\[\|vw\|_{H^{-1}} \leq C \|v\|_{H^{-1}}\|w\|_{H^{1\vee (d/2 +s) }}\]
 where the constant depends on $s$.
\end{lemma}
\begin{proof}
See for example \cite[Theorem 8.1]{BH:21}.
\end{proof}

Then, by additionally using the submultiplicativity of the $H^{1+s}$-norm in two spatial dimensions, by interpolation, and by H\"older's and Jensen's inequalities respectively, we derive for suitable $s\in(0,1)$ and $p>2$
\begin{eqnarray*}
\int_0^T \| u^2 Z_\eps \|^2_{H^{-1}}\, dt 
&\leq &  
C \int_0^T \| u\|^4_{H^{1+s}} \|Z_\eps \|^2_{H^{-1}}\, dt \\
&\leq &  
C \int_0^T \| u\|^{4(1-s)}_{H^1}\|u\|^{4s}_{H^2} \|Z_\eps \|^2_{H^{-1}}\, dt \\
&\leq &  
C \| u\|^{4(1-s)}_{L^{\infty}([0,T],H^1)}
\|u\|^{4s}_{L^2([0,T],H^2)}
\|Z_\eps \|^{2}_{L^{p}([0,T],H^{-1})}
\end{eqnarray*}
provided $s$ is sufficiently small and $p$ sufficiently large.

Similarly, we obtain the estimate
\begin{eqnarray*}
\int_0^T\|u[Z_\eps^2-C_0]\|^2_{H^{-1}}\, dt 
&\leq & 
C \| u\|^{2(1-s)}_{L^{\infty}([0,T],H^1)}
\|u\|^{2s}_{L^2([0,T],H^2)}
\|Z_\eps^2-C_0 \|^{2}_{L^{p}([0,T],H^{-1})}
\end{eqnarray*}

We obtain the following theorem.

\begin{theorem} 
\label{thm:ACCH}
In the setting of the  CH/AC-homotopy in 2D, let $u\in C^0([0,T],H^1)\cap L^2([0,T],H^2)$ be a solution to \eqref{e:limitAC} and fix $\epsilon_0>0$ sufficiently small. 
Then for all $T>0$ and sufficiently large $p>2$ there is a constant $\tilde{K}>0$ such that for all $\eps\in(0,\eps_0)$ 
and for all $\gamma \geq \eps^{1/2}$, we have
\begin{eqnarray*}
\lefteqn{\mathbb{P} \left( \sup_{[0,T]}\|u_\eps-u-Z_\eps\|^2_{L_2} >  \tilde{K} \gamma\right)}\\
&\leq &
\mathbb{P} \left(\eps^{1/4}\sup_{[0,T]}  \|Z_\eps \|_{C^0} > 1/2 \right)
+\mathbb{P} \left(\|u(0)-u_\eps(0)\|^2_{L^2} > \gamma\right)\\
&& +\mathbb{P} \left( \|Z_\eps^2-C_0\|^2_{L^p([0,T],H^{-1})}  > \gamma\right) 
+\mathbb{P} \left( \|Z_\eps\|^2_{L^p([0,T],H^{-1})}  > \gamma\right) 
\\&& +\mathbb{P} \left( \|Z_\eps^3\|^2_{L^2([0,T],H^{-1})}  > \gamma\right) 
+\mathbb{P} \left( \eps^{1/2}\|Z_\eps\|^6_{L^6([0,T],L^6)}  > \gamma\right) 
\end{eqnarray*}
\end{theorem}

Thus, we have finally reduced the whole approximation result to a statement about the stochastic convolution.

\subsection{AC in 2D with higher order regularization}

Recall that for AC with higher order regularization, we have 
\[
A_\eps=-\eps^2\Delta^2+\Delta
\quad \text{and}\quad 
A=\Delta .
\]
Moreover the space $V_\eps'$ is an $H^{-2}$ space with an $\eps$-dependent equivalent norm 
\[
\|w\|^2_{V_\eps'} = \langle (1-\Delta+\eps^2\Delta^2)^{-1} w,w\rangle_{L^2}.
\]
Hence 
\[\|(A_\eps-A)u\|_{V_\eps'} = \eps^2 \|\Delta^2(1-\Delta+\eps^2\Delta^2)^{-1/2}u\|
\leq C \eps \|u\|^2_{H^2} 
\]
and we can proceed with the first term in the residual as before.
The second term vanishes, as $F=F_\eps$.

Moreover, for the third term we use that 
$\|w\|_{V_\eps'} \leq C \|w\|_{H^{-1}}$ and obtain exactly the same estimates as in the CH/AC-homotopy.

Thus we can prove the following.

\begin{theorem}
In the setting of AC with higher order regularization in 2D, let $u\in C^0([0,T],H^1)\cap L^2([0,T],H^2)$ be a solution to \eqref{e:limitAC} and fix $\epsilon_0>0$ sufficiently small. 
Then for  all $T>0$ and sufficiently large $p>2$ there is a constant $\tilde{K}>0$ such that for all $\eps\in(0,\eps_0)$ 
and  all $\gamma \geq \eps$ we have
\begin{eqnarray*}
\lefteqn{\mathbb{P} \left( \sup_{[0,T]}\|u_\eps-u-Z_\eps\|^2_{L_2} >  \tilde{K} \gamma\right)}\\
&\leq &
\mathbb{P} \left(\|u(0)-u_\eps(0)\|^2_{L^2} > \gamma\right) +\mathbb{P} \left( \|Z_\eps^2-C_0\|^2_{L^p([0,T],H^{-1})}  > \gamma\right) 
\\&& +\mathbb{P} \left( \|Z_\eps\|^2_{L^p([0,T],H^{-1})}  > \gamma\right) 
 +\mathbb{P} \left( \|Z_\eps^3\|^2_{L^2([0,T],H^{-1})}  > \gamma\right) 
\end{eqnarray*}
\end{theorem}

\subsection{AC in 2D with regularized noise}

In this case, $A=A_\eps$ and $F=F_\eps$. Thus the first two terms in the residual vanish. Moreover the third term is analogous to the previous two examples, and we recover the same result as in the previous section.

\section{Stochastic Convolution for the CH/AC-homotopy}

\label{sec:stochastic-convolution}

In order to finish the full approximation result, we need to establish various bounds on the stochastic convolution.
We shall briefly state the corresponding results on the flat torus only, i.e.\ for periodic boundary conditions. We assume moreover that the operator $Q_\eps$ is Fourier-diagonal and hence is equal to the operator square root of the covariance operator $Q_\eps^\ast Q_\eps=Q_\eps^2$.

For Neumann boundary conditions on a general domain in particular the quantification of the convergence for $Z_\eps^2$ seems to be an open problem.
So far, we are only aware of the result \cite{GH:19} with Neumann boundary conditions on the square.

On the two-dimensional torus, we  can rely on Fourier series. 
Assuming that $Q_\eps$ is also diagonal in Fourier space, 
we have 
\[
Z_\eps(t)
= \sum_{k\in\Z^2} \alpha_k(\eps) I_k^{(\eps)}(t) e_k 
\quad \text{with}\quad 
I_k^{(\eps)}(t)= \int_0^t e^{-(t-s)\lambda_k(\eps)}\, d\beta_k(s)
\]
for the Fourier basis $e_k$ and complex-valued Brownian motion $\beta_k$ such that $\overline{\beta_k} = \beta_{-k}$.
Moreover, $Q_\eps e_k= \alpha_k(\epsilon) e_k$ and
for our differential operators
$A_\eps e_k= -\lambda_k(\epsilon) e_k$ with $0\leq\lambda_k(\eps)\to\infty$ for $|k|\to\infty$.
We call the $k$th coefficient in the series the $k$th Fourier mode of $Z_\eps$.

\begin{remark}
In contrast to the existing literature our stochastic convolution 
$Z_\eps$ is not stationary. We chose to perform the non-stationary transformation 
to the random PDE \eqref{e:trafo} in order to have the initial condition unchanged. 
But that is not a major issue, and most estimates work exactly the same way. 
Only in a few occasions extra effort is needed.
\end{remark} 

In the example of the CH/AC-homotopy with small space-time white noise, we have  
\[
\lambda_k(\eps) = (1-\eps-\eps|k|^2)|k|^2 
\quad\text{and}\quad 
\alpha_k(\eps)=\sigma_\eps\to 0.
\]
We fix this notation in the sequel so that for the CH/AC-homotopy, we can shall prove the following result related to Theorem \ref{thm:ACCH}.

\begin{prop}
In the setting of the  CH/AC-homotopy in 2D, let $u\in C^0([0,T],H^1)\cap L^2([0,T],H^2)$ be a solution to \eqref{e:limitAC} and fix $\epsilon_0>0$ sufficiently small. 
Then for all $T>0$ and sufficiently large $p>2$, and for all $\gamma \geq \eps^{1/2}$, we have that
\begin{align*}
\lim_{\eps\searrow 0}&\Bigg[\mathbb{P} \left(\eps^{1/4}\sup_{[0,T]}  \|Z_\eps \|_{C^0} > 1/2 \right) +\mathbb{P} \left( \|Z_\eps\|^2_{L^p([0,T],H^{-1})}  > \gamma\right) \\
 &+\mathbb{P} \left( \|Z_\eps^3\|^2_{L^2([0,T],H^{-1})}  > \gamma\right) +\mathbb{P} \left( \eps^{1/2}\|Z_\eps\|^6_{L^6([0,T],L^6)}  > \gamma\right)\Bigg]=0.
\end{align*}
\end{prop}
\begin{proof}
Firstly, it is well-known using the Kolmogorov continuity theorem or the Sobolev embedding of $W^{\alpha,p}$ into $C^0$ that for all $\delta>0$ and $p>1$ that there is a constant $C>0$ such that 
\[
\mathbb{E}\|Z_\eps(t)\|^p_{C^0} \leq C \sigma_\eps^p \left( 1+ \sum_{k\not=0} \frac{|k|^\delta}{\lambda_k(\eps)}  \right)^{p/2}.
\]

We can extend this result, using for example the celebrated factorization method, to obtain that
for all $\delta>0$, $T>0$ and $p>1$ there is a constant $C>0$ such that 
\[
\mathbb{E}\left[\sup_{[0,T]}\|Z_\eps\|^p_{C^0}\right] \leq C \sigma_\eps^p \left( 1+ \sum_{k\not=0} \frac{|k|^\delta}{\lambda_k(\eps)} \right)^{p/2}.
\]   
Thus, we obtain by Chebyshev's inequality
\begin{eqnarray*}
\mathbb{P} \left(\eps^{1/4}\sup_{[0,T]}  \|Z_\eps \|_{C^0} > 1/2 \right)
&\leq& \left( 2\eps^{1/4} \sigma_\eps\right)^{p} \E \left[\sup_{[0,T]}  \|Z_\eps \|_{C^0}^p\right] \\
& \leq& C \eps^{p/4} \sigma_\eps \left( 1+ \sum_{k\not=0} \frac{|k|^\delta}{\lambda_k(\eps)} \right)^{p/2} \\
&\to& 0 \quad \text{for}\quad \eps\to 0.
\end{eqnarray*}
For the convergence above, we can simply evaluate the series by a splitting into $|k|\leq \eps^{-1/2}$ and $|k|> \eps^{-1/2}$ and using integral comparison theorems to obtain
\[
\sum_{k\not=0} \frac{|k|^\delta}{\lambda_k(\eps)} \sim \eps^{\delta/2}
\quad\text{ for }\delta >0 
\qquad\text{and} \qquad  
\sum_{k\not=0} \frac{1}{\lambda_k(\eps)} \sim \log(\eps^{-1}).
\]
We see that, as expected, we are in the critical case in two spatial dimensions, 
and the bound for the $C^0$-norm of $Z_\eps$ diverges slowly for $\eps \to 0$. 

In a similar way, one can show for $\eps\to 0$ that
\[
\mathbb{P} \left( \eps^{1/2}\|Z_\eps\|^6_{L^6([0,T],L^6)}  > \gamma\right)
\to 0
\]
and 
\[
\mathbb{P} \left( \|Z_\eps\|^2_{L^p([0,T],H^{-1})}  > \gamma\right) 
\to 0
\]
where, in the last integral, one relies on an additional factor $1/|k|^2$ appearing due to the $H^{-1}$-norm so that
\[
\E \|Z_\eps\|^2_{H^{-1}} 
\leq C \sigma_\eps^2 \left( 1+ \sum_{k\not=0} \frac{|k|^{-2}}{\lambda_k(\eps)} \right) 
\leq  C \sigma_\eps^2 \left( 1+ \sum_{k\not=0} \frac1{|k|^{4}} \right)
\]
and we can use Gaussianity for higher order moments to yield the result.
\end{proof}

\begin{remark}
In all three cases above and in the two cases below, we could quantify the convergence in more detail and could even make $\gamma$ small and $\eps$-dependent. This would not only give us convergence in probability, but also a rate of convergence.
But for the sake of simplicity of presentation, we are not entering this discussion.
\end{remark}
One can also expand $Z_\eps^2$ in terms of Fourier-modes
\[
Z_\eps^2(t) = \sum_{k\in\Z^2} \sum_{\ell\in\Z^2} I_{k-\ell}^{(\eps)}(t)I_\ell^{(\eps)}(t) e_k
\]
From the literature it is well-known that for convergence of the type
\[
\E\|Z_\eps^2 - C_\eps\|^2_{L^2([0,T],H^{-1})} \to 0 
\]
we need 
\begin{equation}
    \label{def:Ceps}
    C_\eps =\sum_{k\not=0} \frac{\sigma_\eps^2}{2\lambda_k(\eps)}   
\end{equation}
in order to handle the divergent constant terms. 
This divergent term matches the zeroth Fourier mode of $Z_\eps^2$. All other Fourier modes of $Z_\eps^2$ remain finite in the limit $\eps\to0$. 
This is folklore, see for example \cite{GIP:15,DPD02,BR:13}.

Note that we can always neglect the constant mode $k=0$ in $Z_\eps$. 
The stochastic convolution is not stationary in our case, as $\lambda_0(\eps)=0$, but due to $\sigma_\eps\to0$ every fixed Fourier mode of $Z_\eps$ 
disappears in $Z_\eps^2$ in the limit $\eps\to0$. Moreover, equality in \eqref{def:Ceps} is only needed asymptotically for $\eps\to0$ thus we could neglect any finite number of terms in the sum.

We thus define our renormalizing constant to be 
\begin{equation}
\label{def:C0}
C_0 = \lim_{\eps\to0} \sigma_\eps^2 \sum_{k\not=0} \frac{1}{2\lambda_k(\eps)}.
\end{equation}
\begin{remark}
By the discussion above, we need $\sigma(\eps) \sim 1/\log(\eps^{-1})$ in order to have a nontrivial limit.

Moreover, if $\sigma(\eps) \ll 1/\log(\eps^{-1})$ we have $C_0=0$. Our main result still applies, but this is an uninteresting case, as we would just prove that sufficiently small noise disappears in the limit. 

On the other hand, if $\sigma(\eps) \gg 1/\log(\eps^{-1})$, we are in the regime of triviality observed in \cite{HRW:12}. As the constant $C_\eps$ diverges, the limiting equation heuristically has an arbitrarily strong linear damping, and thus $u\equiv 0$ in the end and $u_\eps \approx Z_\eps \to 0$ in $H^{-1}$, which implies triviality. Nevertheless, $u_\eps$ explodes in $L^2$.
\end{remark}

It is now a lengthy and tedious, but relatively straightforward computation to show that 
\begin{equation}\label{eq:Zsquare}  \E\|Z_\eps^2-C_0 \|^2_{L^2([0,T],H^{-1})} \to 0.
\end{equation}
The key in this computation is that $C_0$ together with the small noise strength controls the divergent zeroth Fourier mode in $Z_\eps^2$. 
An expansion of $Z_\eps$ in terms of Fourier series would yield a threefold sum for the norm, where one can rely on cancellations in expectation due to independence of the Fourier modes of $Z_\eps$. 

In the end, we crucially rely on the $\eps$-independent bound of the series 
\[ 
\sum_{k\not=0} \frac1{|k|^2} \sum_{\ell\in\Z^2\setminus\{0,k\}} \frac1{\lambda_\ell\lambda_{k-\ell}}.
\]
All these computations are well-known in case when $Z_\eps$ is a stationary process, but they work exactly the same way here.  
Due to the non-stationarity of $Z_\eps$ additional exponential terms appear in the computations that are usually just discarded via estimates. 
Moreover, the additional time integral in the $L^2([0,T],H^{-1})$-norm helps, as the integral over the exponential terms generate an additional factor $1/\lambda_k(\eps)$ that enhances the convergence of the series.

Let us finally remark that due to Nelson's theorem, see for example \cite[Theorem I.22]{S:74}, we can always bound higher order moments by powers of the second moment which we just discussed.

In a similar way using Fourier-series expansion we  can establish
\[  \E\|Z_\eps^3\|^p_{L^p([0,T],H^{-1})}     \to 0
\] 
but we refrain from giving details of this lengthy calculation, as we are not interested in the precise asymptotics.

Our final result is:
\begin{theorem}
For the CH/AC-homotopy with periodic boundary conditions on a two dimensional domain perturbed by space-time white noise of strength 
$$\sigma_\eps \sim \frac{1}{\log(\eps^{-1})},$$
supposing that  
$\|u_{\eps}(0)- u(0)\|_{L^2}\to 0 $
in probability, we obtain that
\[\|u_{\eps}- u-Z_\eps\|_{L^\infty([0,T],L^2)}\to 0
\quad\text{
in probability,}
\]
where $u\in L^\infty([0,T],H^1)\cap L^2([0,T],H^2)$ is a solution to 
\[\partial_t u =\Delta u +u-u^3 - 3 C_0 u
\]
with $C_0$ defined as in \eqref{def:C0}.
\end{theorem}

We emphasize that for $\sigma_\eps \ll 1/\log(\eps^{-1})$ the result still holds, but in this case with $C_0=0$.

Furthermore, the result is formulated in a weak sense. One
could optimize it and would actually obtain a rate of convergence if the various moments discussed above were computed in more detail.  

Finally, recall that in the case of the theorem above the $L^2$-norm of $Z_\eps$ remains bounded in the limit $\eps\to0$. In the case where $\sigma_\eps\gg\log(\eps^{-1})$, the  $L^2$-norm of $Z_\eps$ diverges to infinity and we are in the case of triviality, where the result above will hold for $u=0$,
i.e. $u_\eps\approx Z_\eps$ diverges in $L^2$, but converges to $0$ in $H^{-1}$.


\begin{thebibliography}{10}

\bibitem{ABK:12}
D.~C. Antonopoulou, D.~Bl\"{o}mker, and G.~D. Karali.
\newblock Front motion in the one-dimensional stochastic {C}ahn-{H}illiard
  equation.
\newblock {\em SIAM J. Math. Anal.}, 44(5):3242--3280, 2012.

\bibitem{AKM:16}
D.~C. Antonopoulou, G.~Karali, and A.~Millet.
\newblock Existence and regularity of solution for a stochastic
  {C}ahn-{H}illiard/{A}llen-{C}ahn equation with unbounded noise diffusion.
\newblock {\em J. Differential Equations}, 260(3):2383--2417, 2016.

\bibitem{BB:16}
I.~Bailleul and F.~Bernicot.
\newblock Heat semigroup and singular {PDE}s.
\newblock {\em J. Funct. Anal.}, 270(9):3344--3452, 2016.
\newblock With an appendix by F. Bernicot and D. Frey.

\bibitem{BMSS:95}
V.~Bally, A.~Millet, and M.~Sanz-Sol\'{e}.
\newblock Approximation and support theorem in {H}\"{o}lder norm for parabolic
  stochastic partial differential equations.
\newblock {\em Ann. Probab.}, 23(1):178--222, 1995.

\bibitem{BH:21}
A.~Behzadan and M.~Holst.
\newblock Multiplication in {S}obolev spaces, revisited.
\newblock {\em Ark. Mat.}, 59(2):275--306, 2021.

\bibitem{B:19}
N.~Berglund.
\newblock {\em An introduction to singular stochastic {PDE}s---{A}llen-{C}ahn
  equations, metastability, and regularity structures}.
\newblock EMS Series of Lectures in Mathematics. EMS Press, Berlin, 2022.

\bibitem{BG:06}
N.~Berglund and B.~Gentz.
\newblock {\em Noise-induced phenomena in slow-fast dynamical systems}.
\newblock Probability and its Applications (New York). Springer-Verlag London,
  Ltd., London, 2006.
\newblock A sample-paths approach.

\bibitem{BH:04}
D.~Bl\"{o}mker and M.~Hairer.
\newblock Multiscale expansion of invariant measures for {SPDE}s.
\newblock {\em Comm. Math. Phys.}, 251(3):515--555, 2004.

\bibitem{BR:13}
D.~Bl\"{o}mker and M.~Romito.
\newblock Local existence and uniqueness for a two-dimensional surface growth
  equation with space-time white noise.
\newblock {\em Stoch. Anal. Appl.}, 31(6):1049--1076, 2013.

\bibitem{BCCH:21}
Y.~Bruned, A.~Chandra, I.~Chevyrev, and M.~Hairer.
\newblock Renormalising {SPDE}s in regularity structures.
\newblock {\em J. Eur. Math. Soc. (JEMS)}, 23(3):869--947, 2021.

\bibitem{BCF:88}
Z.~Brze\'{z}niak, M.~Capi\'{n}ski, and F.~Flandoli.
\newblock A convergence result for stochastic partial differential equations.
\newblock {\em Stochastics}, 24(4):423--445, 1988.

\bibitem{DPD96}
G.~Da~Prato and A.~Debussche.
\newblock Stochastic {C}ahn-{H}illiard equation.
\newblock {\em Nonlinear Anal.}, 26(2):241--263, 1996.

\bibitem{DPD02}
G.~Da~Prato and A.~Debussche.
\newblock Two-dimensional {N}avier-{S}tokes equations driven by a space-time
  white noise.
\newblock {\em J. Funct. Anal.}, 196(1):180--210, 2002.

\bibitem{DPD03}
G.~Da~Prato and A.~Debussche.
\newblock Strong solutions to the stochastic quantization equations.
\newblock {\em Ann. Probab.}, 31(4):1900--1916, 2003.

\bibitem{DPZ:88}
G.~Da~Prato and J.~Zabczyk.
\newblock A note on semilinear stochastic equations.
\newblock {\em Differential Integral Equations}, 1(2):143--155, 1988.

\bibitem{DP-Z:14}
G.~Da~Prato and J.~Zabczyk.
\newblock {\em Stochastic equations in infinite dimensions}, volume 152 of {\em
  Encyclopedia of Mathematics and its Applications}.
\newblock Cambridge University Press, Cambridge, second edition, 2014.

\bibitem{FJL:82}
W.~G. Faris and G.~Jona-Lasinio.
\newblock Large fluctuations for a nonlinear heat equation with noise.
\newblock {\em J. Phys. A}, 15(10):3025--3055, 1982.

\bibitem{FGL:21-2}
F.~Flandoli, L.~Galeati, and D.~Luo.
\newblock Quantitative convergence rates for scaling limit of {SPDE}s with
  transport noise.
\newblock {\em Preprint}, pages 1--34, 2021.
\newblock \url{https://arxiv.org/abs/2104.01740}.

\bibitem{FGL:21-1}
F.~Flandoli, L.~Galeati, and D.~Luo.
\newblock Scaling limit of stochastic 2{D} {E}uler equations with transport
  noises to the deterministic {N}avier-{S}tokes equations.
\newblock {\em J. Evol. Equ.}, 21(1):567--600, 2021.

\bibitem{Funaki:16}
T.~Funaki.
\newblock {\em Lectures on random interfaces}.
\newblock SpringerBriefs in Probability and Mathematical Statistics. Springer,
  Singapore, 2016.

\bibitem{GH:19}
M.~Gerencs\'{e}r and M.~Hairer.
\newblock Singular {SPDE}s in domains with boundaries.
\newblock {\em Probab. Theory Related Fields}, 173(3-4):697--758, 2019.

\bibitem{GT:16}
B.~Gess and J.~M. T\"{o}lle.
\newblock Stability of solutions to stochastic partial differential equations.
\newblock {\em J. Differential Equations}, 260(6):4973--5025, 2016.

\bibitem{GT:22}
G.~Guatteri and G.~Tessitore.
\newblock Singular limit of two-scale stochastic optimal control problems in
  infinite dimensions by vanishing noise regularization.
\newblock {\em SIAM J. Control Optim.}, 60(1):575--596, 2022.

\bibitem{GIP:15}
M.~Gubinelli, P.~Imkeller, and N.~Perkowski.
\newblock Paracontrolled distributions and singular {PDE}s.
\newblock {\em Forum Math. Pi}, 3:e6, 75, 2015.

\bibitem{Ha:14b}
M.~Hairer.
\newblock Singular stochastic {PDE}s.
\newblock In {\em Proceedings of the {I}nternational {C}ongress of
  {M}athematicians---{S}eoul 2014. {V}ol. {IV}}, pages 49--73, Seoul, 2014.
  Kyung Moon Sa.

\bibitem{Ha:14a}
M.~Hairer.
\newblock A theory of regularity structures.
\newblock {\em Invent. Math.}, 198(2):269--504, 2014.

\bibitem{HP:21}
M.~Hairer and E.~Pardoux.
\newblock Fluctuations around a homogenised semilinear random {PDE}.
\newblock {\em Arch. Ration. Mech. Anal.}, 239(1):151--217, 2021.

\bibitem{HRW:12}
M.~Hairer, M.~D. Ryser, and H.~Weber.
\newblock Triviality of the 2{D} stochastic {A}llen-{C}ahn equation.
\newblock {\em Electron. J. Probab.}, 17:no. 39, 14, 2012.

\bibitem{HW:15}
M.~Hairer and H.~Weber.
\newblock Large deviations for white-noise driven, nonlinear stochastic {PDE}s
  in two and three dimensions.
\newblock {\em Ann. Fac. Sci. Toulouse Math. (6)}, 24(1):55--92, 2015.

\bibitem{HR:31}
M.~Heida and M.~R\"{o}ger.
\newblock Large deviation principle for a stochastic {A}llen-{C}ahn equation.
\newblock {\em J. Theoret. Probab.}, 31(1):364--401, 2018.

\bibitem{KvN:11}
M.~Kunze and J.~van Neerven.
\newblock Approximating the coefficients in semilinear stochastic partial
  differential equations.
\newblock {\em J. Evol. Equ.}, 11(3):577--604, 2011.

\bibitem{LR:15}
W.~Liu and M.~R\"{o}ckner.
\newblock {\em Stochastic partial differential equations: an introduction}.
\newblock Universitext. Springer, Cham, 2015.

\bibitem{NC:88}
A.~Novick-Cohen.
\newblock On the viscous {C}ahn-{H}illiard equation.
\newblock In {\em Material instabilities in continuum mechanics ({E}dinburgh,
  1985--1986)}, Oxford Sci. Publ., pages 329--342. Oxford Univ. Press, New
  York, 1988.

\bibitem{OOR:20}
T.~Oh, M.~Okamoto, and T.~Robert.
\newblock A remark on triviality for the two-dimensional stochastic nonlinear
  wave equation.
\newblock {\em Stochastic Process. Appl.}, 130(9):5838--5864, 2020.

\bibitem{Sch:2000}
G.~Schimperna.
\newblock Abstract approach to evolution equations of phase-field type and
  applications.
\newblock {\em J. Differential Equations}, 164(2):395--430, 2000.

\bibitem{Show}
R.~E. Showalter.
\newblock {\em {M}onotone operators in {B}anach space and nonlinear partial
  differential equations}.
\newblock Mathematical surveys and monographs, Amer. Math. Soc., Providence,
  1997.

\bibitem{S:74}
B.~Simon.
\newblock {\em The {$P(\phi)_{2}$} {E}uclidean (quantum) field theory}.
\newblock Princeton Series in Physics. Princeton University Press, Princeton,
  N.J., 1974.

\bibitem{T:97}
R.~Temam.
\newblock {\em Infinite-dimensional dynamical systems in mechanics and
  physics}, volume~68 of {\em Applied Mathematical Sciences}.
\newblock Springer-Verlag, New York, second edition, 1997.

\bibitem{TW:20}
P.~Tsatsoulis and H.~Weber.
\newblock Exponential loss of memory for the 2-dimensional {A}llen-{C}ahn
  equation with small noise.
\newblock {\em Probab. Theory Related Fields}, 177(1-2):257--322, 2020.

\bibitem{Zabczyk:89}
J.~Zabczyk.
\newblock On large deviations for stochastic evolution equations.
\newblock In {\em Stochastic systems and optimization ({W}arsaw, 1988)}, volume
  136 of {\em Lect. Notes Control Inf. Sci.}, pages 240--253. Springer, Berlin,
  1989.

\end{thebibliography}
\end{document}